\documentclass{amsart}
\usepackage{ amsmath, amsthm, amsfonts, amssymb, color}
 \usepackage{mathrsfs}
\usepackage{amsfonts, amsmath}
\usepackage{amsmath}
\usepackage{amsfonts}
\usepackage{amssymb}
\usepackage{times}

\allowdisplaybreaks

\begin{document}
\addtolength{\parskip}{8pt}

\newcommand\C{{\mathbb C}}

\newtheorem{thm}{Theorem}[section]
\newtheorem{prop}[thm]{Proposition}
\newtheorem{cor}[thm]{Corollary}
\newtheorem{lem}[thm]{Lemma}
\newtheorem{lemma}[thm]{Lemma}
\newtheorem{exams}[thm]{Examples}
\theoremstyle{definition}
\newtheorem{defn}[thm]{Definition}
\newtheorem{rem}[thm]{Remark}
\newcommand\RR{\mathbb{R}}
\newcommand{\la}{\lambda}
\def\RN {\mathbb{R}^n}
\newcommand{\norm}[1]{\left\Vert#1\right\Vert}
\newcommand{\abs}[1]{\left\vert#1\right\vert}
\newcommand{\set}[1]{\left\{#1\right\}}
\newcommand{\Real}{\mathbb{R}}
\newcommand{\supp}{\operatorname{supp}}
\newcommand{\card}{\operatorname{card}}
\renewcommand{\L}{\mathcal{L}}
\renewcommand{\P}{\mathcal{P}}
\newcommand{\T}{\mathcal{T}}
\newcommand{\A}{\mathbb{A}}
\newcommand{\K}{\mathcal{K}}
\renewcommand{\S}{\mathcal{S}}
\newcommand{\blue}[1]{\textcolor{blue}{#1}}
\newcommand{\red}[1]{\textcolor{red}{#1}}
\newcommand{\Id}{\operatorname{I}}
\newcommand\wrt{\,{\rm d}}
\def\SH{\sqrt {H}}

\newcommand{\rn}{\mathbb R^n}
\newcommand{\de}{\delta}
\newcommand{\tf}{\tfrac}
\newcommand{\ep}{\epsilon}
\newcommand{\vp}{\varphi}

\newcommand{\mar}[1]{{\marginpar{\sffamily{\scriptsize
        #1}}}}
\newcommand{\li}[1]{{\mar{LY:#1}}}
\newcommand{\el}[1]{{\mar{EM:#1}}}
\newcommand{\as}[1]{{\mar{AS:#1}}}
\newcommand\CC{\mathbb{C}}
\newcommand\NN{\mathbb{N}}
\newcommand\ZZ{\mathbb{Z}}
\renewcommand\Re{\operatorname{Re}}
\renewcommand\Im{\operatorname{Im}}
\newcommand{\mc}{\mathcal}
\newcommand\D{\mathcal{D}}
\newcommand{\al}{\alpha}
\newcommand{\nf}{\infty}
\newcommand{\comment}[1]{\vskip.3cm
	\fbox{%
		\color{red}
		\parbox{0.93\linewidth}{\footnotesize #1}}
	\vskip.3cm}

\newcommand{\disappear}[1]

\numberwithin{equation}{section}
\newcommand{\chg}[1]{{\color{red}{#1}}}
\newcommand{\note}[1]{{\color{green}{#1}}}
\newcommand{\later}[1]{{\color{blue}{#1}}}
\newcommand{\bchi}{ {\chi}}

\numberwithin{equation}{section}
\newcommand\relphantom[1]{\mathrel{\phantom{#1}}}
\newcommand\ve{\varepsilon}  \newcommand\tve{t_{\varepsilon}}
\newcommand\vf{\varphi}      \newcommand\yvf{y_{\varphi}}
\newcommand\bfE{\mathbf{E}}

\title[ A sharp regularity estimate for the Schr\"odinger  propagator]
{A sharp regularity estimate for the Schr\"odinger\\[2pt] propagator on the  sphere
}

 \author{Xianghong Chen}
  \author{Xuan Thinh Duong}
   \author{Sanghyuk Lee}
 \author{Lixin Yan}
 \address{Xianghong Chen, Department of Mathematics, Sun Yat-sen
 University, Guangzhou, 510275, P.R. China}
 \email{chenxiangh@mail.sysu.edu.cn}
  \address{Xuan Thinh Duong, Department of Mathematics, Macquarie University, NSW 2109, Australia}
\email{xuan.duong@mq.edu.au}
 \address{Sanghyuek Lee,  School of Mathematical Sciences, Seoul National University, Seoul 151-742, Repulic of Korea}
\email{shklee@snu.ac.kr}
 \address{Lixin Yan, Department of Mathematics, Sun Yat-sen   University,
 Guangzhou, 510275, P.R. China}
 \email{mcsylx@mail.sysu.edu.cn}

\date{\today}

\begin{abstract}
Let $\Delta_{\mathbb S^n}$ denote  the Laplace-Beltrami operator on  the $n$-dimensional unit sphere $\mathbb S^n$.
In this paper we show that
$$
\| e^{it  \Delta_{\mathbb S^n}}f \|_{L^4([0, 2\pi) \times \mathbb S^n)}
\leq C \| f\|_{W^{\alpha, 4} (\mathbb S^n)}
$$
holds 
provided that 
$n\geq 2$, $\alpha> {(n-2)/4}.$   
The range of $\alpha$ is sharp up to the endpoint.  
As a consequence, we  obtain space-time estimates for the Schr\"odinger  propagator  $e^{it \Delta_{\mathbb S^n}}$
on the $L^p$ spaces
for  $2\leq p\leq \infty.$
We also prove that  for
zonal functions
on ${\mathbb S}^n$, the Schr\"odinger maximal operator $\sup_{0\leq t<2\pi} |e^{it\Delta_{\mathbb S^n}} f|$
is bounded from $W^{\alpha, 2}(\mathbb S^n) $ to $L^{\frac{6n}{3n-2}}(\mathbb S^n)$
whenever
$\alpha>{1/ 3}$.
 \end{abstract}

\maketitle

 %  \tableofcontents

\section{Introduction}
 \setcounter{equation}{0}

 \noindent
 Let $\mathbb S^n$ denote the $n$-dimensional unit sphere
  in ${\mathbb R}^{n+1}  $ endowed with the  standard metric. Denote by  $ \Delta_{\mathbb S^n}$
  the Laplace-Beltrami operator   on
$\mathbb S^n$.
For $k=0, 1, \cdots,$
denote by ${\mathscr H}^n_k$ the space of spherical harmonics of degree $k$
(for background on the spherical harmonics, cf. \cite[Chapter IV]{SteinWeiss}). It is well-known that
one has  
the orthogonal decomposition
$$
L^2(\mathbb S^n)=\bigoplus_{k=0}^{\infty}{\mathscr H}^n_k;
$$
  moreover,
  $$
   \Delta_{\mathbb S^n} Y_k =-k (k+n-1)Y_k, \ \  \ \forall Y_k\in {\mathscr H}^n_k,
  $$
 and  ${\mathscr H}^n_k$ is of dimension
 $\sim k^{n-1}.$
  Denote by
  $$
 {\mathbb P}^n_k: L^2(\mathbb S^n)\to {\mathscr H}^n_k
  $$
  the orthogonal projection from $L^2(\mathbb S^n) $ to $ {\mathscr H}^n_k$.

In this paper we  study   regularity properties of solutions to the Cauchy  problem for  the Schr\"odinger equation
 on   ${\mathbb S^n}$:
\begin{eqnarray}\label{e1.1}
  i{\partial_t u } + \Delta_{\mathbb S^n} u=0,  \ \ \ \ \ \
 u(0, x)=f(x),
\end{eqnarray}
where the unknown $u(t, x)$ is a complex-valued function on
$[0, 2\pi) \times\mathbb S^n.$ %
For convenience, we willl write $\mathbb T=[0, 2\pi)$.  
    By spectral theory,
the solution operator
for  
equation \eqref{e1.1}
is  
given by
 \begin{eqnarray}\label{e1.2}
e^{it \Delta_{\mathbb S^n}}f:=\sum_{k=0}^\infty e^{-itk(k+n-1)} {\mathbb P}^n_k (f).
\end{eqnarray}
By the Parseval identity, the Schr\"odinger operator
  $e^{it \Delta_{\mathbb S^n}}$  acts
  isometrically  
  on   $L^2(\mathbb S^n)$.
For $2\leq  p<\infty$, it follows   by the Sobolev embedding
  that,
  with 
  $ s(p, n)=n\big({1/2}- {1/p}\big)$,  
  \begin{eqnarray}
\label{e1.3}
\big\| e^{it \Delta_{\mathbb S^n}}f \big\|_{L^p( \mathbb S^n)}
   \leq    C\|   e^{it \Delta_{\mathbb S^n}}f \|_{ W^{s(p, n), 2}(\mathbb S^n)  }
 &= & C\|   f \|_{ W^{s(p, n), 2}(\mathbb S^n)  }
  \nonumber\\
  &\leq &  C'\|   f \|_{ W^{s(p, n), p}(\mathbb S^n)  },
\end{eqnarray}
where  the Sobolev space  $W^{s, p} $  is
defined 
by
$$W^{s, p}(\mathbb S^n)=\{f:  \|(I-\Delta_{\mathbb S^n} )^{\frac s2} f\|_{p}<\infty \}.$$ %
 Taking the $L^p(\mathbb T)$ norm
 then gives  
  \begin{align}
\label{e1.4}
\big\| e^{it \Delta_{\mathbb S^n}}f \big\|_{L^p(\mathbb T \times \mathbb S^n)}\leq
C\|   f \|_{ W^{s(p, n), p}(\mathbb S^n).}  
\end{align}
Note that  
estimate \eqref{e1.4} does not
take into account possible gain of  
provided by the average  
over  
$\mathbb T$.
 
 In contrast to  the fixed time estimate \eqref{e1.3},
 it
 is of interest  
  to  seek  the minimal   $\alpha$ for which the bound
  \begin{align}
\label{e1.5}
\big\| e^{it \Delta_{\mathbb S^n}}f \big\|_{L^p(\mathbb T \times \mathbb S^n)}\leq
C\|   f \|_{ W^{\alpha, p}(\mathbb S^n)  }
\end{align}
holds.  
On  the
circle  
$\mathbb S^1$,    it is
known  
that
\eqref {e1.5}  holds for $\alpha=0$
 when  
 $2\leq p\leq 4$,
 for  
 $\alpha>0$
 when  
 $4<p\leq 6$,  
 and
 for  
 $\alpha>{1/2}-{3/p}$
 when  
 $6<p<\infty$,
 by  
 a bound of Zygmund \cite{Z}:  
 \begin{align}
 \label{e1.1z}
\big\| e^{it \Delta_{\mathbb S^1}}f \big\|_{L^4(\mathbb T \times \mathbb S^1)} \leq
 C\|   f \|_{ L^{2}(\mathbb S^1)  },
\end{align}
and  the following well-known inequality due to  Bourgain \cite{Bourgain1993}: % ,
 \begin{align}
\label{e1.b3}
\big\| e^{it \Delta_{\mathbb S^1}}f \big\|_{L^6(\mathbb T \times \mathbb S^1)}\leq C
 \|   f \|_{ W^{\varepsilon, 2}(\mathbb S^1)  }, \ \ \ \ \forall \varepsilon>0.
\end{align}
For the sphere $\mathbb S^n, n\geq 2$,
    it is remarkable that  in \cite[Theorem 4]{BGT2004},  Burq-G\'erard-Tzvetkov     used  the clustering property
	of the spectrum of the Laplace-Beltrami operator
	and $L^2$-$L^4$ norm of spectral projections of the Laplace associated to finite intervals of high frequencies
	 to establish  the following
  Strichartz estimates:
\begin{align}
\label{e1.6}
\big\| e^{it \Delta_{\mathbb S^n}}f \big\|_{L^4(\mathbb T \times \mathbb S^n)}
\leq C  \|  f \|_{ W^{\alpha, 2}(\mathbb S^n)  }, \ \ \ \  \alpha>\alpha(4, n),
\end{align}
where $\alpha(4, n)$ is given by
\begin{eqnarray}\label{e1.7}
\alpha(4, n)=
\left\{
\begin{array}{llll}
{1\over 8}, \ \ \ &{\rm if}\ \ \  &n=2;\\[6pt]
  {n-2\over 4},\ \ \ &{\rm if}\ \ \  &n\geq 3.
\end{array}
\right.
\end{eqnarray}
The loss of $\alpha$  derivatives in the  estimate \eqref{e1.6} is essentially
sharp in the sense that similar estimates fail with $\alpha\leq \alpha(4, n)$  if $n\geq 3$ (resp. $\alpha<\alpha(4, 2)$ if $n=2$).
From \eqref{e1.6}, one infers
 that  for  $p=4$, estimate \eqref {e1.5}  averaging over time $\mathbb T $ yields a gain
${3/8} $ derivatives for $n=2$;    and a gain
${1/2}$ derivatives for $n\geq 3.$

Our first  goal in this paper  is  to prove
an  
$L^4$-estimate     
  with a loss of  
  $\varepsilon>0$   derivative  on  
  $\mathbb S^2$;
 and
 a loss of  
 of ${(n-2)/4} $ derivatives on  
 $\mathbb S^n$,  
 $ n\geq 3.$
  More precisely, we have the following result.

\begin{thm}
\label{Thm1.1}\   Let $n\geq 2.$ The solution $e^{it  \Delta_{\mathbb S^n}}f$ of \eqref{e1.1} satisfies
 \begin{eqnarray}\label{e1.8}
 \| e^{it  \Delta_{\mathbb S^n}}f \|_{L^4(\mathbb T \times \mathbb S^n)}
 \leq C \| f\|_{W^{\alpha, 4} (\mathbb S^n)}, \ \ \ \ \ \  \alpha> {n-2\over 4}.
  \end{eqnarray}
 Moreover, \eqref{e1.8}
 fails  
 when $\alpha < {(n-2)/4}$.
 \end{thm}

As a consequence of  Theorem~\ref{Thm1.1},
we 
  deduce the following space-time  estimates
 for the Schr\"odinger  propagator $e^{it\Delta_{\mathbb S^n}}$
 on the  
 $L^p$ spaces,  
 $2\leq p\leq \infty$.  
 This result is optimal
when $n=2$ and $2\leq p\leq \infty$,
and when
 $n\geq 3$ and $4\leq p\leq \infty$.

\begin{cor}
\label{cor1.2}\  Let  $n\geq 2$ and
$2\leq p\leq \infty$. Then the following estimate 
\begin{eqnarray}\label{e1.9}
 \left\|  e^{it\Delta_{\mathbb S^n}} f\right\|_{L^p( \mathbb T\times \mathbb S^n)} \leq C    \|f\|_{W^{\alpha, p}(\mathbb S^n)} 
 \end{eqnarray}
holds for $\alpha>\alpha_0(p, n)$, where $\alpha_0(p, 2)= \max \{ 0, {1- 4/p}  \}$  
  and,  
  for $n\geq 3$,
\begin{eqnarray}\label{e1.10}
\alpha_0(p, n)=
\left\{
\begin{array}{llll}
  n\big({1\over 2}-{1\over p}\big)+ {2\over p} -1, \ \ \ &{\rm if}
&2\leq p\leq 4;\\[6pt]
  n\big({1\over 2}-{1\over p}\big)- {2\over p},\ \ \ &{\rm if}
&4<p\leq \infty.
\end{array}
\right.
\end{eqnarray}
 Conversely,   if  \eqref{e1.9}  holds,
 then
 $\alpha\ge \max \{ 0, n({1/ 2}-{1/ p}) -{2/ p} \}.$
 \end{cor}

  From \eqref{e1.6} and \eqref{e1.7}, the proof of \eqref{e1.8}  in
  Theorem~\ref{Thm1.1} reduces  to show  it for the case $n=2$, whose
   proof  combines several arguments: firstly,
	using a number-theoretic argument in Burq-G\'{e}rard-Tzvetkov \cite{BGT2005},
   	we reduce the $L^4_t$ norm (the norm in $t$) to the $L^2_t$ norm. This makes use of the fact that
   	the eigenvalues of $\Delta_{\mathbb S^2}$ essentially take quadratic values (such an argument goes back to
   	earlier work on $\Lambda(p)$ sets in the 60's).
  Secondly, using the spectral information, we apply an
  almost orthogonality result of Kadec \cite{Kadec1964} to reduce the
  $L^2_t$ norm of the Schr\"{o}dinger evolution $e^{it\Delta}f$ to that of
  the half-wave evolution $e^{it\sqrt{- \Delta_{\mathbb S^2}}}f$.
  This allows us to apply an $L^4_x\rightarrow L^4_x L^2_t$ local smoothing estimate for the half-wave group
  $e^{it\sqrt{- \Delta_{\mathbb S^2}}}$ due to Mockenhaupt-Seeger-Sogge \cite{MSS1993} to conclude the proof
  of \eqref{e1.8} for $n=2$.

  The sharpness of  Theorem~\ref{Thm1.1} 
   is proved by using a  semiclassical
  dispersion estimate of Burq-G\'{e}rard-Tzvetkov \cite{BGT2004} as well as ideas from Rogers \cite{Rogers2008}.
  We remark that the same argument can be used to show that, on a general compact manifold of dimension $n$,
  a necessary condition for the analogue of \eqref{e1.8} to hold is $\alpha\ge\max\{0, n\big({1/2}-{1/ p}\big)- {2/p}\}$;
  in particular, sharp regularity estimates on $\mathbb T^2$   follow immediately from the Strichartz
  estimates in Bourgain \cite{Bourgain1993} and Bourgain-Demeter \cite{BourgainDemeter2015} for the Schr\"{o}dinger equation.

\smallskip

   In the second part of the paper, we   consider  
   the problem of identifying the values  
   $\alpha$   for which
\begin{eqnarray}\label{e1.11}
\left\| \sup_{0\leq t< 2\pi} |e^{it\Delta}f| \right\|_{L^q(\mathbb S^n)} \leq C(\alpha) \| f\|_{W^{\alpha, 2}(\mathbb S^n)}
\end{eqnarray}
holds for some $2\leq q<\infty$.
Inequality  
\eqref{e1.11} has implications
on  
the existence almost everywhere of $\lim_{t\to 0} u(t,x)$   for solutions $u$ of the
Schr\"odinger equation \eqref{e1.1}.
On the circle $\mathbb S^1$,   Moyua and Vega \cite{MoyuaVega2008} 
used \eqref{e1.b3} and the Sobolev embedding to   
show 
that  
\eqref{e1.11} 
holds with  
$\alpha>{1/3}$ 
and  
 $q=6$.  
They  
also point out that  $\alpha\geq {1/ 4}$ is a necessary condition for \eqref{e1.11} 
to be true.
For the sphere $\mathbb S^n, n\geq 2$,
    Wang and Zhang \cite {WZ}  proved  \eqref{e1.11} with 
    $\alpha>{1/2}$ 
    and $q=2$ 
    by taking
full advantage of the spectrum concentration.

The second goal of this paper is to investigate  \eqref{e1.11} for  
zonal functions  
on  
$\mathbb S^n$. 
We will show the following. 

 \begin{thm}
\label{Thm1.2}\    Let $n\geq 2$. For any $\alpha>{1/3}$,  there exists a constant
$C=C(\alpha)>0$ such that
 for any  zonal
 polynomial $f$,   
\begin{eqnarray}\label{e1.12}
\left\| \sup_{0\leq t< 2\pi} |e^{it \Delta_{\mathbb S^n}}f| \right\|_{L^{\frac{6n}{3n-2}}(\mathbb S^n)}
 \leq C \| f\|_{W^{\alpha,2}(\mathbb S^n)}.
\end{eqnarray}
 Moreover,  
 estimate \eqref{e1.12} fails when   $\alpha<{1/4}$, even if the left-hand size is replaced by the $L^1$ norm.
 \end{thm}

   The proof of Theorem~\ref{Thm1.2} utilizes asymptotic formulas for the zonal spherical harmonics to reduce
   the estimates to the setting of $\mathbb S^1$. More precisely, we use an asymptotic formula from Szeg\"o \cite{Sz}
   to expand (up to an error term) the zonal function away from the poles into a modulated cosine series. The sufficiency
   part then follows from a maximal inequality for the Schr\"odinger equation on $\mathbb S^1$ and the Sobolev embedding
    (the latter is used to bound the maximal function near the poles). The proof of the necessity part is an adaptation of
	 a counterexample on $\mathbb S^1$ due to Moyua-Vega \cite{MoyuaVega2008}, who used the Gauss sum to show that $s\ge {1/ 4}$
	 is necessary for the Schr\"odinger maximal function on $\mathbb S^1$ to be bounded. In our case, by writing the cosine
	 function as a sum of conjugate exponentials, we are led to consider how conjugate initial data are evolved under
	 the Schr\"odinger equation (in particular, how do they add for specific initial data). By examining the argument
	 of the value of the Gauss sum, we manage to show that Moyua-Vega's counterexample carries over to $\mathbb S^n$,
	 with the blowup occuring on a set of possibly smaller measure.

    As a consequence of Theorem~\ref{Thm1.2}, we have the following result.
	
    \begin{cor}
\label{cor1.3}\   Let $n\geq 2$. For any $\alpha>{1/3}$,   the solution $u(t,x)$ to equation
 \eqref{e1.1} converges pointwise to the initial data $f$, whenever  a zonal function  $f
 $ belongs to $ W^{\alpha, 2} (\mathbb S^n)$ for $\alpha>{1/ 3}.$
 \end{cor}

We would like to mention that in the Euclidean setting,   Carleson \cite{Ca}
  proposed  the problem of identifying the optimal $s$ for which
$$\lim_{t\to 0}e^{it\Delta} f(x) = f(x), \ \ \ \ \ \ {\rm a.e.}\ x\in\mathbb R^n
$$
 whenever  $f\in H^{\alpha}({\mathbb R}^n)$. In dimension one, Carleson \cite{Ca}
 proved convergence for    ${\alpha}\geq {1/4}$ and
 Dahlberg and Kenig \cite{DK}  showed that this   is sharp.
 In higher dimensions, this problem  was recently  settled
 by Du-Guth-Li \cite{DGL}   for $n=2$ and  ${\alpha}>{1/3}$;   and  by Du-Zhang \cite{DZ} for $n\geq 3$ and
 ${\alpha}>{n/ 2(n+1)}$.
   Due to a counterexample by Bourgain \cite{B}, up to the endpoint,
    these two latter results are sharp.
	In contrast to the case of $\mathbb R^n$, very little is known
	for  the sphere $\mathbb S^n$.	Even for the case of  $\mathbb S^1$,
	as pointed out  by Moyua and Vega \cite{MoyuaVega2008},
 	 the strategy of Carleson\cite{Ca} gives a worse result
	than the case on the real line  ${\mathbb R}$, and
	    we do not know whether $f\in W^{\alpha, 2}(\mathbb S^1)$ for ${\alpha}\geq {1/4}$
	is sufficient yet.
		It would be interesting to establish  the sharp version  of \eqref{e1.11} to find 
		 the optimal  value ${\alpha}$ for which \eqref{e1.11} holds.

The paper is organized as follows.
In Section 2 we 
first  
prove 
the $L^4$-estimate stated in  
Theorem~\ref{Thm1.1}. 
Then,  
by analytic interpolation, we obtain 
the  
$L^p(\mathbb T \times \mathbb S^n)$- 
estimates  
of $e^{it \Delta_{\mathbb S^n}}f$ 
stated  
in Corollary~\ref{cor1.2}, for 
$2\leq p\leq \infty$. 
 The  proof of    Theorem~\ref{Thm1.2}
 is   
 given in   Section 3.

   \medskip

\section{Proof of Theorem~\ref{Thm1.1}   }
 \setcounter{equation}{0}

To prove  \eqref{e1.8}  in
  Theorem~\ref{Thm1.1}, from \eqref{e1.6} and \eqref{e1.7} it suffices  to show it  for $n=2$,
  whose proof  is based on the following Lemma~\ref{le2.1} and Lemma~\ref{le2.2}.

\begin{lemma}
\label{le2.1} For any $\alpha>0$, there exists a constant $C=C(\alpha)$ independent of $f$ such that
\begin{align}
\label{e2.2}
\| e^{it \Delta_{\mathbb S^2}}f \|_{L^4(\mathbb T)}
 \leq C  \left(\sum_{k=0}^{\infty}  \big|(1- \Delta_{\mathbb S^2})^{\alpha/2} {\mathbb P}^2_k(f)\big|^2 \right)^{1/2}.
\end{align}
\end{lemma}

\begin{proof} 
The proof of Lemma~\ref{le2.1} is inspired by the result of
Burq-G\'{e}rard-Tzvetkov \cite[Proposition~3.1]{BGT2005} (with $u_0=v_0$).
From \eqref{e1.2}, we write
\begin{align*}
 | e^{it \Delta_{\mathbb S^2}}f |^2=
  \sum_{k,\ell=0}^{\infty}
e^{-i t[k(k+1)+\ell(\ell+1)]} {\mathbb P}^2_k(f){\mathbb P}^2_\ell(f).
\end{align*}
By the Parseval identity,
we get
\begin{eqnarray}\label{e2.3}
\left\| e^{it \Delta_{\mathbb S^2}}f \right\|_{L^4(\mathbb T)}^4
&=&\Big\|\sum_{s=0}^{\infty}
e^{-i ts} \sum_{k(k+1)+\ell(\ell+1)=s}{\mathbb P}^2_k(f){\mathbb P}^2_\ell(f)
\Big\|_{L^2(\mathbb T)}^2\nonumber \\
&=& 2\pi \sum_{s=0}^{\infty}
\Big| \sum_{k(k+1)+\ell(\ell+1)=s} {\mathbb P}^2_k(f){\mathbb P}^2_\ell(f) \Big|^2\nonumber \\
&\leq & 2\pi \sum_{s=0}^{\infty} r(s)
\sum_{k(k+1)+\ell(\ell+1)=s} \big| {\mathbb P}^2_k(f){\mathbb P}^2_\ell(f) \big|^2,
\end{eqnarray}
where
$$
r(s)=\#\big\{(k,l): k(k+1)+\ell(\ell+1)=s\big\}.
$$
Notice that
$$
k(k+1)+\ell(\ell+1)=s \iff
 (2k+1)^2 + (2\ell+1)^2 = 4s+2.
 $$
It follows from classical results (see \cite[Theorem~278]{HardyWright}) on the sum of squares function that  for any $\alpha>0$,
there exists a constant $C=C(\alpha)>0$
such that
$$
r(s)\leq C (1+s)^\alpha.
$$
Consequently, one can bound  \eqref{e2.3}  by
$$\sum_{s=0}^{\infty} (1+s)^{\alpha}
\sum_{k(k+1)+\ell(\ell+1)=s} \big| {\mathbb P}^2_k(f){\mathbb P}^2_\ell(f) \big|^2.$$
Since $s=k(k+1)+\ell(\ell+1)$, this can be bounded by
$$\sum_{s=0}^{\infty}
\sum_{k(k+1)+\ell(\ell+1)=s}  \big(1+k(k+1)\big)^{\alpha}\big(1+\ell(\ell+1)\big)^{\alpha}
\big| {\mathbb P}^2_k(f){\mathbb P}^2_\ell(f) \big|^2,$$
which equals
$$\left(\sum_{k=0}^{\infty}  |(I- \Delta_{\mathbb S^2})^{\alpha/2} {\mathbb P}^2_k(f)|^2\right)
\left(\sum_{\ell} |(I- \Delta_{\mathbb S^2})^{\alpha/2} {\mathbb P}^2_\ell(f)|^2\right).$$
Therefore,
$$
\| e^{it \Delta_{\mathbb S^2}}f \|_{L^4(\mathbb T)}^4
\leq C
\left(\sum_{k=0}^{\infty} |(I- \Delta_{\mathbb S^2})^{\alpha/2} {\mathbb P}^2_k(f)|^2\right)^2.
$$
This proves Lemma~\ref{le2.1}.
\end{proof}

Recall that the half-wave group on $L^2(\mathbb S^2)$ is defined by
\begin{eqnarray}\label{e2.4}
e^{it\sqrt{- \Delta_{\mathbb S^2}}}(f)=\sum_{k=0}^\infty e^{it\sqrt{k(k+1)}}{\mathbb P}^2_k(f).
\end{eqnarray}
 Lemma~\ref{le2.1}  provides a useful way to relate the Schr\"odinger group and
the half-wave group as in the following.

\begin{lemma}\label{le2.2}
We have
$$
\left(\sum_{k=0}^{\infty} |{\mathbb P}^2_k(f)|^2 \right)^{1/2}
\approx \left\| e^{it\sqrt{- \Delta_{\mathbb S^2}}}(f)-{\mathbb P}^2_0(f) \right\|_{L^2(\mathbb T)}+|{\mathbb P}^2_0(f)|.
$$
\end{lemma}

\begin{proof}
Recall that an exponential system
$\{e^{i\lambda_kt}\}$ ($\lambda_k\in\mathbb R$) is said to be a Riesz sequence in $L^2(\mathbb T)$ if for any coefficients $\{c_k\}$,
$$
\Big\|\sum_{k=0}^{\infty} c_k e^{i\lambda_kt}\Big\|_{L^2(\mathbb T)}\approx
\Big(\sum_{k=0}^{\infty} |c_k|^2 \Big)^{1/2}.$$
A celebrated theorem of Kadec \cite{Kadec1964} implies that $\{e^{i\lambda_kt}\}$ forms a Riesz sequence in $L^2(\mathbb T)$ provided
$$\sup_k |\lambda_k-k|<\frac14.$$
By modulation, it is easy to see that the same conclusion holds if
$$\sup_k \Big|\lambda_k-k-\frac12\Big|<\frac14.$$

Take $\lambda_k=\sqrt{k(k+1)}$, $k\ge 1$. By direct checking, we see that
$$\left| \sqrt{k(k+1)}-k-\frac12\right|=
\frac{\frac14}{\sqrt{k(k+1)}+k+\frac12}\le \frac 18<\frac14.$$
Thus the last condition is satisfied, and so
$$\Big\|\sum_{k\ge 1} c_k e^{it\sqrt{k(k+1)}}\Big\|_{L^2(\mathbb T)}\approx
\Big(\sum_{k\ge1} |c_k|^2 \Big)^{1/2}.$$
With $c_k={\mathbb P}^2_k(f)$, we obtain
\begin{align*}
\Big(\sum_{k=0}^{\infty} |{\mathbb P}^2_k(f)|^2 \Big)^{1/2}
&\approx
\Big(\sum_{k\ge1} |{\mathbb P}^2_k(f)|^2 \Big)^{1/2}+|{\mathbb P}^2_0(f)|\\
&\approx \left\| e^{it\sqrt{- \Delta_{\mathbb S^2}}}(f)-{\mathbb P}^2_0(f)  \right\|_{L^2(\mathbb T)}+|{\mathbb P}^2_0(f)|.
\end{align*}
This proves Lemma~\ref{le2.2}.
\end{proof}

 Now we start to prove our main result, Theorem~\ref{Thm1.1}.

\begin{proof}[Proof of
  Theorem~\ref{Thm1.1}]
    To  prove
 \eqref{e1.8},    from \eqref{e1.6} and \eqref{e1.7} it suffices to show it for
  $n=2$. In this case, an essential observation   is   to
  reduce to the following
 local smoothing result by Mockenhaupt-Seeger-Sogge \cite[Theorem~6.2]{MSS1993}
 (with $X=Y=\mathbb S^2$, $\mathcal F=e^{it\sqrt{- \Delta_{\mathbb S^2}}}(1- \Delta_{\mathbb S^2})^{-\alpha/2}$, $p=q=4$):
For any $\alpha>0$, there exists a constant $C=C(\alpha)$
\begin{align}
\label{e2.5}
\Big\| e^{it\sqrt{- \Delta_{\mathbb S^2}}}f \Big\|_{L^4(\mathbb S^2;L^2(\mathbb T))}
\leq C \|(1- \Delta_{\mathbb S^2})^{\alpha/2}f\|_{L^4(\mathbb S^2)}.
\end{align}
Indeed, we apply Lemma~\ref{le2.2} with $f$ replaced by $(1- \Delta_{\mathbb S^2})^{\alpha/2}(f)$ to obtain
\begin{eqnarray*}
 &&\hspace{-0.6cm}\Big(\sum_{k=0}^{\infty}  |(1- \Delta_{\mathbb S^2})^{\alpha/2} {\mathbb P}^2_k (f)|^2 \Big)^{1/2}\notag \\
&\approx &\Big\| e^{it\sqrt{- \Delta_{\mathbb S^2}}}(1- \Delta_{\mathbb S^2})^{\alpha/2}f-{\mathbb P}^2_0(f)
\Big\|_{L^2(\mathbb T)}+|{\mathbb P}^2_0(f)|\notag\\
&\leq &C \Big\| e^{it\sqrt{- \Delta_{\mathbb S^2}}}(1- \Delta_{\mathbb S^2})^{\alpha/2}f \Big\|_{L^2(\mathbb T)}+C|{\mathbb P}^2_0(f)|.
\end{eqnarray*}
This, in combination with  Lemma~\ref{le2.1}, yields
\begin{eqnarray*}
\big\| e^{it \Delta_{\mathbb S^2}}f \big\|_{L^4(\mathbb T \times \mathbb S^2)}
\leq& C\Big\| e^{it\sqrt{- \Delta_{\mathbb S^2}}}(1- \Delta_{\mathbb S^2})^{\alpha/2}f \Big\|_{L^4(\mathbb S^2;L^2(\mathbb T))}
 +C\|{\mathbb P}^2_0(f)\|_{L^4(\mathbb S^2)}.
\end{eqnarray*}
Estimate \eqref{e1.8} then  follows readily from \eqref{e2.5}.  This proves the sufficiency part of
 Theorem~\ref{Thm1.1}.
\end{proof}

\begin{rem} 
	Note that if the Sobolev norm on the right-hand side of \eqref{e2.5} is based on $L^2$,
	Cardona and Esquivel \cite{CE} obtained  sharp $W_x^{\alpha,2}\rightarrow L^q_xL^2_t$ estimate on $\mathbb S^n$ 
	with the critical exponent $q= {2(n+1)}/{(n-1)}$ by using purely the $L^2\to L^p$
	spectral estimates for the operator norm of the spectral projections associated to the spherical harmonics proved in \cite{KL}.
	%See also Remark~\ref{rem2.7} below.
	Using the method in the proof of Theorem~\ref{Thm1.1} above, their result can be deduced from a special case 
	of \cite[Theorem 3.2]{MSS1993} regarding the corresponding estimate for the half-wave equation, and vice versa. 
\end{rem}

   The sharpness of the range $\alpha$ in \eqref{e1.8} in Theorem~\ref{Thm1.1}
   is a special case of the following more general proposition.

	\begin{prop}\label{prop2.3}
	Let $n\ge1$ and let  $  \Delta_{\mathbb S^n}$ be the Laplace-Beltrami operator on  $\mathbb S^n$.
Suppose $p\geq 2$ and for some number $\alpha$ it holds that
\begin{align}
\label{e2.8}
\| e^{it  \Delta_{\mathbb S^n}}f \|_{L^p(\mathbb T \times \mathbb S^n)}
\leq C  \| f\|_{W^{\alpha, p}(\mathbb S^n)}.
\end{align}
	Then
$$\alpha\geq \max\Big\{0,\,  n\big({1\over 2}-{1\over p}\big)- {2\over p}\Big\}.
$$
	\end{prop}
	
	\begin{proof}
The necessity of $\alpha\ge 0$ can be easily seen by taking $f=Z^n_k, k\rightarrow\infty$ and noting that
$\|e^{it  \Delta_{\mathbb S^n}}f \|_{L^p(\mathbb S^n)}=\|f \|_{L^p(\mathbb S^n)}$ in this case.
The necessity of   $\alpha\ge n\big({1/2}-{1/ p}\big)- {2/p}$ will be shown below using a semiclassical
 dispersion estimate of Burq-G\'{e}rard-Tzvetkov \cite{BGT2004} and ideas from Rogers \cite[Section~2]{Rogers2008}.
More precisely, Lemma 2.5 (with $M=\mathbb S^n$) of \cite{BGT2004} implies that there exists a bump
function $0\le\varphi\in C_c^\infty(\mathbb R)$, such that for all sufficiently small $h>0$,
$$
\big\|e^{-ih  \Delta_{\mathbb S^n}}\varphi(h\sqrt{-  \Delta_{\mathbb S^n}})\big\|_{L^\infty(\mathbb S^n)}\leq C {h^{-n/2}}.
$$
Fix $x\in\mathbb S^n$, and let
$$
f(y)=(1-  \Delta_{\mathbb S^n})^{-\alpha/2} e^{-ih  \Delta_{\mathbb S^n}}\varphi(h\sqrt{-  \Delta_{\mathbb S^n}})(x,y).
$$
Then
\begin{align}
\|(1-  \Delta_{\mathbb S^n})^{\alpha/2}f\|_{L^p(\mathbb S^n)}
&=\big\|e^{-ih  \Delta_{\mathbb S^n}}\varphi(h\sqrt{-  \Delta_{\mathbb S^n}})\big\|_{L^p(\mathbb S^n)}\notag\\
&\leq C  \big\|e^{-ih  \Delta_{\mathbb S^n}}\varphi(h\sqrt{-  \Delta_{\mathbb S^n}})\big\|_{L^\infty(\mathbb S^n)}\notag\\
&\leq C h^{-n/2}.\label{eq:upper-bound}
\end{align}
On the other hand, we have
\begin{align*}
e^{it  \Delta_{\mathbb S^n}}f(y)
&=e^{i(t-h)  \Delta_{\mathbb S^n}}{(1-  \Delta_{\mathbb S^n})^{-\alpha/2}}{\varphi(h\sqrt{-  \Delta_{\mathbb S^n}})}(x,y)\\
&=\sum_{k\ge 0} e^{-i(t-h)k(k+n-1)}
\frac{\varphi\big(h\sqrt{k(k+n-1)}\big)}{(1+k(k+n-1))^{\alpha/2}} Z^n_k(x,y).
\end{align*}
Due to the support property of $\varphi$, the sum above is over the $k$'s with
$h \sqrt{k(k+n-1)}\lesssim 1$.
 Consequently, we can find a small constant $c>0$ so that
$$|t-h|\le c h^2  \Longrightarrow \Re\big(e^{-i(t-h)k(k+n-1)}\big) \ge 1/2.$$
With a possibly smaller $c$, we also have
$$
d(x,y)\le c k^{-1} \Longrightarrow Z^n_k(x,y)\geq c k^{n-1}.
$$
Therefore, when $|t-h|\le c h^2$ and $d(x,y)\le c h$ (change $c$ again if necessary),
$$\Re \big(e^{it  \Delta_{\mathbb S^n}}f\big)\geq c  h^{-(n-\alpha)}.$$
It follows that
\begin{equation}
\label{eq:lower-bound}
\| e^{it  \Delta_{\mathbb S^n}}f \|_{L^p(\mathbb T\times \mathbb S^n)}^p
\geq c h^{2}h^{n}h^{-(n-\alpha)p}.
\end{equation}
Combining \eqref{eq:upper-bound} and \eqref{eq:lower-bound}, we see that for \eqref{e1.8} to hold, we must have
$$\frac{n+2}{p}-(n-\alpha)\ge -\frac n2,$$
that is, $\alpha\ge n\big({1/ 2}-{1/ p}\big)- {2/ p}$.
%The proof of     Proposition~\ref{prop2.3} is complete.
\end{proof}

By the
 Sobolev emmbedding (see for example, \cite[(9), p.315]{Tr}),  we have
 $$
 L^{\infty}(\mathbb T, W^{{n\over 2}+\varepsilon, 2} (\mathbb S^n)) 
 \hookrightarrow L^{\infty}(\mathbb T \times \mathbb S^n), \ \ \ \forall \varepsilon>0.
 $$
To show Corollary~\ref{cor1.2}, we need  the following result.

\begin{lemma}
\label{le2.3}
Let $n\ge1$ and let  $  \Delta_{\mathbb S^n}$ be the Laplace-Beltrami operator on  $\mathbb S^n$.
 For every  $y\in{\mathbb R}$ and $\varepsilon>0$, there exists a constant $C=C(\varepsilon, n)$ independent of $y$
 such that
\begin{align}
\label{e2.8}
\| e^{it  \Delta_{\mathbb S^n}}f \|_{L^\infty(\mathbb T \times \mathbb S^n)}
\leq C  \|(1-  \Delta_{\mathbb S^n})^{iy+ \frac{n}{4}+\varepsilon}f\|_{L^\infty(\mathbb S^n)}.
\end{align}
\end{lemma}

\begin{proof}
 It suffices to show that the operator
${e^{it  \Delta_{\mathbb S^n}}}{(1-  \Delta_{\mathbb S^n})^{-iy-\frac{n}{4}-\varepsilon}}$
is uniformly bounded on $L^\infty(\mathbb S^n)$.
To show it, let us first recall
some properties of the zonal spherical harmonic functions (see for example \cite{SteinWeiss, Sz}).
Let $C^{\lambda}_k(t)$ be the Gegenbauer polynomial of degree $k$ and index $\lambda,$  i.e,
$$
C^{\lambda}_k(t)={\Gamma(\lambda+{1\over 2}) \over \Gamma(2\lambda)} {\Gamma(k+2\lambda)\over \Gamma(k+\lambda+{1\over 2})}
P_k^{\lambda-{1\over 2}, \lambda-{1\over 2}}(t),
$$
 where $P_k^{\alpha, \beta}$ is the Jacobi polynomial of degree $k$ (see  \cite[p. 80]{Sz}). Denote by $|x-y|\in [0, \pi]$
 the great-circle distance between $x$ and $y$ on $\mathbb S^n$. It is a standard fact that
 $$
{\mathbb P}^n_k f(x) =\int_{\mathbb S^n} Z^n_k(x, y) f(y) dy,
 $$
 where $Z^n_k(x, y)$ is the zonal spherical harmonic function of degree $k$, given by
 $$
 Z^n_k(x,y)={k+\lambda\over \lambda} C_k^{\lambda}(\cos |x-y|).
 $$
 Note that $\cos |x-y|$ represents the inner product in ${\mathbb R}^{n+1}$.

 Denote by $Z^n_k(x,y)$ the zonal spherical harmonic of degree $k$. We can write
$$\frac{e^{it  \Delta_{\mathbb S^n}}}{(1-  \Delta_{\mathbb S^n})^{iy+\frac{n}{4}+\varepsilon}}
=\sum_{k=0}^{\infty} \frac{e^{-itk(k+n-1)}}{\big(1+k(k+n-1)\big)^{iy+\frac{n}{4}+\varepsilon}} Z^n_k(x,y).$$
For any fixed $x$, we have
$$\left\|\frac{e^{it  \Delta_{\mathbb S^n}}}{(1-  \Delta_{\mathbb S^n})^{ iy+\frac{n}{4}+\varepsilon}}\right\|^2_{L^2(\mathbb S^n)}
=\sum_{k=0}^{\infty} \frac{\|Z^n_k\|^2_{L^2(\mathbb S^n)}}{\big(1+k(k+n-1)\big)^{\frac{n}{2}+2\varepsilon}}.
$$
Since $\|Z^n_k\|^2_{L^2(\mathbb S^n)}\leq C k^{n-1}$ (\cite[p.140]{SteinWeiss}),
it follows that the kernel of the operator ${e^{it  \Delta_{\mathbb S^n}}}{(1-  \Delta_{\mathbb S^n})^{-iy-\frac{n}{4}-\varepsilon}}$ is
 uniformly bounded in $L^2(\mathbb S^n)\subset L^1(\mathbb S^n)$, thus defines a uniformly bounded operator on $L^\infty(\mathbb S^n)$.
 \end{proof}

 \begin{rem}
In \cite{T}, M. Taylor studied the Schr\"odinger equation on the spheres at times that are rational multiples of $\pi$ to
  that  for all $1<p<\infty$ and all $s\in {\mathbb R}$. It is shown that
 \begin{eqnarray}\label{bbbbb}
 e^{-\pi i ({m/k}) \Delta_{\mathbb S^n} }:  W^{s, p}(\mathbb S^n) \to W^{s-(n-1)|{1\over 2}-{1\over p}|, p}(\mathbb S^n)
   \end{eqnarray}
   extends to a bounded operator.
   Such estimates also hold in the endpoint cases $p=1, \infty$,
   with $L^1$ replaced by the local Hardy space and $L^{\infty}$ replaced by bmo.
   For the sharpness of this estimate \eqref{bbbbb}, we   refer the reader to Taylor \cite[page 148]{T}.
   \end{rem}

  In the end of this section, we give a proof of   Corollary~\ref{cor1.2}.

   \begin{proof}[Proof of Corollary~\ref{cor1.2}] In view of  Proposition~\ref{prop2.3},
 the necessity part of Corollary~\ref{cor1.2} follows readily.

	For \eqref{e1.9},
let us  first prove it for   $n\geq 3$. The case $4<p<\infty$
    is treated by    interpolating  between  \eqref{e1.8} and the case $p=\infty$.
	 Given any   $\varepsilon>0$,   we consider the analytic family of operators
\begin{eqnarray}\label{e2.222}
T_z= e^{z^2} (1-\Delta_{\mathbb S^n})^{-{  n-2+4\varepsilon +(n+2)z \over 4}} e^{it \Delta_{\mathbb S^n}}, \ \ \ \ \  0 <{\rm Re}\ \! z\leq 1.
\end{eqnarray}
For $z=iy, {\rm Re}\ \! z=0$,   we apply \eqref{e1.8} of Theorem~\ref{Thm1.1} to obtain
that the operators
$$
T_{iy}= e^{-y^2} (1-\Delta_{\mathbb S^n})^{-{  n-2+4\varepsilon +i(n+2)y  \over 4}}  e^{it \Delta_{\mathbb S^n}}
$$
are bounded from $L^4(  \mathbb S^n)$ into $L^4(\mathbb T\times \mathbb S^n)$ and there exists a constant   $C$ independent of $y$ such that
\begin{eqnarray*}
 \|T_{iy} f \|_{L^4(\mathbb T \times \mathbb S^n)}
&=&e^{-y^2} \big\| (1-\Delta_{\mathbb S^n})^{-{   n-2+4\varepsilon\over 4}} e^{it \Delta_{\mathbb S^n}}
[(1-\Delta_{\mathbb S^n})^{-{ i(n+2)y  \over 4}} f]\big\|_{L^4(\mathbb T \times \mathbb S^n)}\nonumber\\
&\leq & Ce^{-y^2}\big\|  (1-\Delta_{\mathbb S^n})^{-{ i(n+2)y\over 4}} f \big\|_{L^4(  \mathbb S^n)}
\nonumber\\
&\leq &
C  \|    f  \|_{L^4(  \mathbb S^n)}
\end{eqnarray*}
since $ \|  (1-\Delta_{\mathbb S^n})^{-{ i(n+2)y/4}}   \big\|_{L^4(  \mathbb S^n)\to L^4(  \mathbb S^n)}
\leq C    (1+|y|)^{(n+2)/2  } $ (see \cite[Theorem 3.1]{DOS}).
On the other hand, we apply Lemma~\ref{le2.3} to get
\begin{eqnarray*}
 \|T_{1+iy} f \|_{L^{\infty}(\mathbb T \times \mathbb S^n)}
&=& e^{1-y^2}\big\|  (1-\Delta_{\mathbb S^n})^{-{ 2n+  4\varepsilon+iy(n+2)  \over 4}} e^{it \Delta_{\mathbb S^n}}
f\big\|_{L^{\infty}(\mathbb T \times \mathbb S^n)} \\
 &\leq&   C  \| f  \|_{L^{\infty}(  \mathbb S^n)}
\end{eqnarray*}
with $C$ independent of $y.$ Then by the complex interpolation theorem (see \cite[Theorem 3.4, pp. 151-152]{CT}),
\begin{eqnarray}\label{ppp}
 \|T_{\theta} f \|_{L^{p}(\mathbb T \times \mathbb S^n)}\leq  C\| f \|_{L^{p}(  \mathbb S^n)}
\end{eqnarray}
where $0\leq \theta \leq 1$ and $1/p=(1-\theta)/4.$ This gives  \eqref{e1.9}  for all $4\leq p\leq \infty$.

Now for   $2\leq p\leq 4$,  we     consider the analytic family of operators
$
T_z= e^{z^2} (1-\Delta_{\mathbb S^n})^{-{  [ (n-2)z +4\varepsilon ]/4}} e^{it \Delta_{\mathbb S^n}},
$ $
0 <{\rm Re}\ \! z\leq 1.
$
   Interpolating between   \eqref{e1.8} of Theorem~\ref{Thm1.1}  and \eqref{e1.4} (with $p=2$) yields
   \eqref{e1.9}  for all $2\leq p\leq 4$ by making  a minor modification to the proof of \eqref{ppp}.
   This  proves \eqref{e1.9} for  the case $n= 3$.

 The proof of  \eqref{e1.9} for  the case $n= 2$  is similar to that of the case $n\geq 3$, and we omit the detail here.
  This completes  the proof of  Corollary~\ref{cor1.2}.
  \end{proof}

%\begin{rem} 
	%When $n\ge 5$ and $2 < p< \frac{2(n+1)}{n-1}$, improvement on Corollary \ref{cor1.2} can be obtained 
	%from the $W_x^{\alpha,2}\rightarrow L^q_xL^q_t$ estimate in \cite[Theorem 1.1]{CE}.  
%\end{rem}

 \begin{rem}\label{rem2.7} 
 In   \cite[Theorem 1.1]{CE},  Cardona and Esquivel showed that for $n\geq 2$  
  and $p, q$ satisfying $2\leq p\leq \infty$ and $2\leq q< \infty,$  
the following
  estimate 
\begin{eqnarray}
\label{e2.76}
\big\| e^{it \Delta_{\mathbb S^n}}f \big\|_{L_x^p(\mathbb S^n, L_t^q(\mathbb T))}
\leq C  \|  f \|_{ W^{\alpha, 2}(\mathbb S^n)  } 
\end{eqnarray}
holds for  $\alpha>\aleph ({p, q, n})$, where 
\begin{eqnarray}\label{e2.77} \hspace{1cm}
\aleph ({p, q, n})=
\left\{
\begin{array}{llll}
 {n-1\over 2}\big( {1\over 2}- {1\over p}\big) + \big({1\over 2}-{1\over q}\big), \ \ \ &{\rm if}\ \ \  &2\leq p\leq {2(n+1)\over n-1};\\[6pt]
n\big( {1\over 2}- {1\over p}\big) -{1\over q},\ \ \ &{\rm if}\ \ \  & p> {2(n+1)\over n-1}.
\end{array}
\right.
\end{eqnarray}
The regularity order $\aleph ({p, 2, n})$ is sharp in any dimension $n$, in the sense that \eqref{e2.76} does not hold
 for all $\alpha <\aleph ({p, 2, n}).$

 From \eqref{e1.10} and \eqref{e2.77}, we see that when $n\ge 6$ and $2 < p< 3$, improvement on Corollary \ref{cor1.2} can be obtained 
	from the $W_x^{\alpha,2}\rightarrow L^p_xL^p_t$ estimate as in \eqref{e2.76} above. That is, 	
	 \eqref{e1.9}  holds provided that $n\ge 6$, $ \alpha>\aleph({p, p, n})$  and $2 < p< 3$.
\end{rem}

 \medskip

 \section{Proof of Theorem \ref{Thm1.2}}
 \setcounter{equation}{0}

 To prove \eqref{e1.12} in Theorem~\ref{Thm1.2}, we need
 a slight variant of a Strichartz estimate on $\mathbb T$ due to Bourgain \cite[Proposition~2.36]{Bourgain1993}.

\begin{lemma}For any $\varepsilon>0$, integer $N\ge 1$, and numerical sequence $\boldsymbol{a}=\{a_k\}$, there exists a
 constant $C=C(\varepsilon)$ independent of $N$ and  $\boldsymbol{a}$ such that
\begin{equation}\label{e4.1}
\left\|\sum_{k=0}^{N-1} a_k e^{-itk(k+n-1)}e^{\pm ik\theta}\right\|_{L^6(\mathbb T\times \mathbb T)}
\leq C(\varepsilon) N^{\varepsilon}\|\boldsymbol{a}\|_{\ell^2}.
\end{equation}
\end{lemma}

\begin{proof}
We consider only the case $e^{\pm ik\theta}=e^{- ik\theta}, k=0, \cdots, N-1$. 
The proof for the case $e^{\pm ik\theta}=e^{ ik\theta}$ is similar.
For convenience, we will write $m=n-1$.

Following the proof of \cite[Proposition~2.36]{Bourgain1993}, the $L^6$-norm in \eqref{e4.1}
can be written out using Plancherel's theorem as
\begin{equation}\label{e4.2}
\left\|\sum_{k=0}^{N-1} a_k e^{-itk(k+m)}e^{-ik\theta}\right\|_{L^6(\mathbb T^2)}^6
=c \sum_{u,v}  \left|\sum_{j+k+\ell=u\atop j(j+m)+k(k+m)+\ell(\ell+m)=v} a_j a_k a_\ell  \right|^2,
\end{equation}
where $c$ is an absolute constant. Note that since $0\le j,k,\ell<N$,
in the last sum $u\lesssim N, v\lesssim N^2$. Denote
$$
r_{u,v}^{(m)}=\#\{(j,k,\ell)\in\mathbb N^3:j+k+\ell=u,\ j(j+m)+k(k+m)+\ell(\ell+m)=v\}.
$$
By the Cauchy-Schwarz inequality, \eqref{e4.2} can be bounded by
$$
\sum_{u,v}  r_{u,v}^{(m)}
\sum_{j+k+\ell=u\atop j(j+m)+k(k+m)+\ell(\ell+m)=v} |a_j|^2 |a_k|^2 |a_\ell|^2.$$
Since
$$\sum_{u,v}\sum_{j+k+\ell=u\atop j(j+m)+k(k+m)+\ell(\ell+m)=v} |a_j|^2 |a_k|^2 |a_\ell|^2 = \|\boldsymbol{a}\|_{\ell^2}^6,
$$
to prove \eqref{e4.1} it suffices to show
\begin{eqnarray}\label{e4.41}
r_{u,v}^{(m)}\le C(\varepsilon) N^\varepsilon,\ \forall\varepsilon>0.
\end{eqnarray}
When $m=0$, \eqref{e4.41} has been shown to hold in the proof of \cite[Proposition~2.36]{Bourgain1993}. On the other hand, by definition,
 $r_{u,v}^{(m)}=r_{u,v-mu}^{(0)}$. It follows immediately that \eqref{e4.41} also holds for $m\ge 1$. This completes the proof of the lemma.
\end{proof}

By a standard argument (see Moyua-Vega \cite[Proposition~1]{MoyuaVega2008}), \eqref{e4.1} implies the following maximal inequality.\\

\begin{cor}\label{cor4.2}
  For any $\varepsilon>0$, integer $N\ge 1$, and numerical sequence $\boldsymbol{a}=\{a_k\}$,
  there exists a constant $C(\varepsilon) $ such that
\begin{equation}\label{e4.3}
\left\|\sup_{t\in\mathbb T}\Big|\sum_{k=0}^{N-1} a_k e^{-itk(k+n-1)}e^{\pm ik\theta}\Big|\right\|_{L^6(\mathbb T)}
\leq C(\varepsilon) N^{\frac13+\varepsilon}\|\boldsymbol{a}\|_{\ell^2}.
\end{equation}
\end{cor}

\begin{proof}[Proof of   Theorem~\ref{Thm1.2}]
First, let us show  \eqref{e1.12} in Theorem~\ref{Thm1.2} by using \eqref{e4.3}
 to deduce a maximal inequality for zonal functions on $\mathbb S^n$.
Suppose $f(y)$ is zonal with respect to $x_0\in\mathbb S^n$. Let
$$\widetilde Z_k(y)=\frac{Z_k(x_0,y)}{\|Z_k(x_0,\cdot)\|_2}.$$
By the assumption on $f$, we can write
$$f(y)=\sum_{k=0}^{2N-1} a_k \widetilde Z_k(y)$$
for some dyadic $N\ge1$ and coefficients $a_k\in\mathbb C$.
By considering a dyadic decomposition, we may (and will) further assume that
$$f(y)=\sum_{k=N}^{2N-1} a_k \widetilde Z_k(y).$$
Write $q=\frac{6n}{3n-2}$ and
$$\langle x_0,y \rangle_{\mathbb R^{n+1}}=\cos\theta,\ 0\le\theta\le\pi.$$
Since $Mf:=\sup_{0\le t<2\pi} |e^{it\Delta}f|$ is also zonal with respect to $x_0$, we can write
$$\|Mf\|_{L^q(\mathbb S^n)}
\approx\left(\int_0^\pi |Mf(\theta)|^q (\sin\theta)^{n-1} d\theta\right)^{1/q}.$$
By the Sobolev embedding, we have
$$\|Mf\|_{L^\infty}\le C N^{\frac n2} \|\boldsymbol{a}\|_{\ell^2}.$$
Therefore, for any fixed $c>0$,
\begin{eqnarray}\label{e4.4}
\left(\int_0^{cN^{-1}}+\int_{\pi-cN^{-1}}^{\pi}\right)
|Mf(\theta)|^q (\sin\theta)^{n-1} d\theta
\le C N^{\frac q3} \|\boldsymbol{a}\|_{\ell^2}^q.
\end{eqnarray}

In the region $cN^{-1}\le\theta\le\pi-cN^{-1}$, by Theorem 8.21.13 of Szeg\H{o} \cite{Sz}, we have the uniform estimate
\begin{equation}\label{e4.5}
\widetilde Z_k(\theta)
= \frac{c_k}{(\sin\theta)^{\frac{n-1}{2}}} \cos\left(\Big(k+\frac{n-1}{2}\Big) \theta-\frac{n-1}{4}\pi\right)
+ \frac{O(1)}{k(\sin\theta)^{\frac{n+1}{2}}},
\end{equation}
where $c_k>0$ is a constant bounded above and below.
Correspondingly, we can bound
$$
\sup_{0\leq t< 2\pi} |e^{it\Delta}f(\theta)|\le M^{(0)}f(\theta)+M^{(1)}f(\theta),
$$
where $M^{(0)}f(\theta)$ is given by
\begin{align*}
\frac{1}{(\sin\theta)^{\frac{n-1}{2}}}\sup_{0\le t<2\pi}\left|\sum_{k=N}^{2N-1} c_k a_k e^{-itk(k+n-1)}
\cos\left(\Big(k+\frac{n-1}{2}\Big) \theta-\frac{n-1}{4}\pi\right)\right|
\end{align*}
and, after applying the Cauchy-Schwarz inequality,
\begin{equation}\label{e4.6}
M^{(1)}f(\theta)\le C \frac{\|\boldsymbol{a}\|_{\ell^2}}{\sqrt{N}(\sin\theta)^{\frac{n+1}{2}}}.
\end{equation}
From \eqref{e4.6} it follows immediately that
\begin{equation}\label{e4.7}
\int_{cN^{-1}}^{\pi-cN^{-1}} |M^{(1)}f(\theta)|^q(\sin\theta)^{n-1} d\theta
\le C N^{\frac q3}\|\boldsymbol{a}\|_{\ell^2}^q.
\end{equation}
On the other hand, writing $\cos x=\frac12 \sum_{\pm} e^{\pm ix}$, we can bound
$$
M^{(0)}f(\theta)\le \frac{1}{2(\sin\theta)^{\frac{n-1}{2}}} \sum_{\pm} \sup_{0\le t<2\pi}
\left|\sum_{k=N}^{2N-1} c_k a_k e^{-itk(k+n-1)} e^{\pm i k\theta}\right|.
$$
Thus, by H\"older's inequality, we have
\begin{align*}
&\int_{cN^{-1}}^{\pi-cN^{-1}} |M^{(0)}f(\theta)|^q (\sin\theta)^{n-1} d\theta \\
\le C &
(1+\log N)^{1-\frac{q}{6}}
\left\|\sum_{\pm} \sup_{0\le t<2\pi}\Big|\sum_{k=N}^{2N-1} c_k a_k e^{-itk(k+n-1)} e^{\pm i k\theta}\Big|\right\|_{L^6(\mathbb T)}^q.
\end{align*}
Applying \eqref{e4.3} to the right-hand side, we see that, for any $\varepsilon>0$,
\begin{equation}\label{e4.8}
\int_{cN^{-1}}^{\pi-cN^{-1}} |M^{(0)}f(\theta)|^q (\sin\theta)^{n-1} d\theta
\le C(\varepsilon) N^{\frac q3+\varepsilon} \|\boldsymbol{a}\|_{\ell^2}^q.
\end{equation}

Combining \eqref{e4.4}, \eqref{e4.7} and \eqref{e4.8}, we obtain
$$
\left(\int_0^\pi |Mf(\theta)|^q (\sin\theta)^{n-1} d\theta\right)^{1/q}
\le C(\varepsilon) N^{\frac13+\varepsilon} \|\boldsymbol{a}\|_{\ell^2},
$$
which completes the proof since
$$
N^{\frac13+\varepsilon} \|\boldsymbol{a}\|_{\ell^2}
\approx\|f\|_{H^{\frac13+\varepsilon}(\mathbb S^n)}.
$$
This completes the proof of the sufficiency  part of Theorem~\ref{Thm1.2}.

\smallskip

Next we prove the   necessity  part  of Theorem \ref{Thm1.2}.
The unboundedness will be shown using zonal polynomials. The proof is based on a counterexample
on $\mathbb S^1$ due to Moyua and Vega \cite{MoyuaVega2008}.
Fix $x_0\in\mathbb S^n$. Let $N\ge1$ be a large integer and consider
$$f_N(y)=\sum_{k=0}^{N-1} \frac{\widetilde Z_k(y)}{c_k},$$
where $\widetilde Z_k(y)$ and $c_k$ are as above (with $c_0:=1$).
Since $\|f_N\|_{H^s(\mathbb S^n)}\lesssim N^{\frac12+s}$, it  suffices to show that
$$Mf_N\ge C N^{\frac34}$$
holds on a set $E_N$ of measure $\mu(E_N)\ge C>0$.

As before, write $\langle x_0,y \rangle_{\mathbb R^{n+1}}=\cos\theta,\ 0\le\theta\le\pi$.
Fix a small $\varepsilon>0$. By \eqref{e4.5}, we have the uniform estimate
$$\frac{\widetilde Z_k(y)}{c_k}
= \frac{1}{(\sin\theta)^{\frac{n-1}{2}}} \cos\left(\Big(k+\frac{n-1}{2}\Big) \theta-\frac{n-1}{4}\pi\right)
+ \frac{O(1)}{k+1},\ \varepsilon\le\theta\le\pi-\varepsilon.$$
Therefore, when $\varepsilon\le\theta\le\pi-\varepsilon$,
$$
e^{it\Delta}f_N (\theta)
= \frac{1}{(\sin\theta)^{\frac{n-1}{2}}} \sum_{k=0}^{N-1} e^{-itk(k+n-1)}\cos\left(k\theta+\phi_n(\theta)\right)
+ O(\log N),$$
where $\phi_n(\theta):=\frac{n-1}{2}\theta-\frac{n-1}{4}\pi.$ Denote the sum above by $S(t,\theta)$, i.e.
$$S(t,\theta)
=\sum_{k=0}^{N-1} e^{-itk(k+n-1)}\cos\left(k\theta+\phi_n(\theta)\right).$$
By writing $\cos x=\frac12 \sum_{\pm} e^{\pm ix}$, we have
\begin{align}
S(t,\theta)
&= \frac12 \sum_{\pm} \left(\sum_{k=0}^{N-1} e^{-itk(k+n-1)}e^{\pm ik\theta} \right) e^{\pm i\phi_n(\theta)} \notag \\
&=: \frac12\sum_{\pm} S_{\pm}(t,\theta) e^{\pm i\phi_n(\theta)}.\label{eq:S_pm}
\end{align}

Suppose $t=\frac{2\pi}{q}$ with $q\approx\sqrt N$ being an odd integer. Then
$$S_+\Big(\frac{2\pi}{q},\theta\Big)=\sum_{k=0}^{N-1} e^{-2\pi i\frac{k(k+n-1)}{q}} e^{ik\theta}.$$
Since $e^{-2\pi i\frac{k(k+n-1)}{q}}$ is $q$-periodic in $k$, we can write
\begin{align*}
S_+\Big(\frac{2\pi}{q},\theta\Big)
&=\left(\sum_{k=0}^{q\big\lfloor\frac N q\big\rfloor-1}
+\sum_{k=q\big\lfloor\frac N q\big\rfloor}^{N-1}\right)
e^{-2\pi i\frac{k(k+n-1)}{q}} e^{ik\theta}\\
&=\left(\sum_{\ell=0}^{\big\lfloor\frac N q\big\rfloor-1} e^{i\ell q\theta}\right)
\left(\sum_{k=0}^{q-1} e^{-2\pi i\frac{k(k+n-1)}{q}} e^{ik\theta}\right)
+O(q)\\
&=:\left(\sum_{\ell=0}^{\big\lfloor\frac N q\big\rfloor-1} e^{i\ell q\theta}\right)
s_+\Big(\frac{2\pi}{q},\theta\Big)+O(q).
\end{align*}
If $\theta=\frac{2\pi p}{q}$ for some integer $p$, then
\begin{align*}
s_+\Big(\frac{2\pi}{q},\frac{2\pi p}{q}\Big)
&=\sum_{k=0}^{q-1} e^{-2\pi i\frac{k(k+n-1)}{q}} e^{2\pi i\frac{kp}{q}}\\
&=\sum_{k=0}^{q-1} e^{-2\pi i\frac{k^2+(n-1-p)k}{q}}.
\end{align*}
The last sum is a Gauss sum and can be evaluated explicitly to give
$$s_+\Big(\frac{2\pi}{q},\frac{2\pi p}{q}\Big) = \omega_q \sqrt q\, e^{2\pi i \frac{r(n-1-p)^2}{q}}, $$
where $r$ is an integer such that $4r\equiv 1$ (mod $q$), and
$$\omega_q=
\begin{cases}
     1& \text{if } q\equiv 1 \text{ (mod 4)},\\
     -i& \text{if } q\equiv 3 \text{ (mod 4)}.
\end{cases}$$
If $p$ is even, we can further write
\begin{align*}
s_+\Big(\frac{2\pi}{q},\frac{2\pi p}{q}\Big)
&=\omega_q \sqrt q\, e^{2\pi i \frac{r(n-1)^2+rp^2}{q}}
e^{-2\pi i \frac{2r(n-1)p}{q}}\\
&=\omega_q \sqrt q\, e^{2\pi i \frac{r(n-1)^2+rp^2}{q}}
e^{- i \frac{\pi (n-1) p}{q}}.
\end{align*}

Suppose $\theta=\frac{2\pi p}{q}+\eta$ with $p$ even and $|\eta|\le \frac{\pi}{8N}$. Then
\begin{align*}
\left|s_+\Big(\frac{2\pi}{q},\theta\Big)-s_+\Big(\frac{2\pi}{q},\frac{2\pi p}{q}\Big) \right|
&=\left| \sum_{k=0}^{q-1} e^{-2\pi i\frac{k(k+n-1)}{q}} e^{ik\frac{2\pi p}{q}} (e^{ik\eta}-1) \right|\\
&\le \sum_{k=0}^{q-1} | e^{ik\eta}-1 |
 \le \sum_{k=0}^{q-1} k |\eta|
 \le \frac{q^2}{N}.
\end{align*}
Thus,
$$s_+\Big(\frac{2\pi}{q},\theta\Big)
=s_+\Big(\frac{2\pi}{q},\frac{2\pi p}{q}\Big)+O\left(\frac{q^2}{N}\right).$$
Combining this with
$$e^{i\phi_n(\theta)}
=e^{i\phi_n(\frac{2\pi p}{q})}
+O\left(\frac1N\right),$$
we see that
\begin{align*}
&\hspace{-0.5cm}S_+\Big(\frac{2\pi}{q},\theta\Big)\, e^{i\phi_n(\theta)}\\
=&\left(\sum_{\ell=0}^{\big\lfloor\frac N q\big\rfloor-1} e^{i\ell q\theta}\right)
s_+\Big(\frac{2\pi}{q},\frac{2\pi p}{q}\Big) e^{i\phi_n(\frac{2\pi p}{q})} +O(q)\\
=&\left(\sum_{\ell=0}^{\big\lfloor\frac N q\big\rfloor-1} e^{i\ell q\eta}\right)
\left(\omega_q \sqrt q\, e^{2\pi i \frac{r(n-1)^2+rp^2}{q}}
e^{- i \frac{\pi (n-1) p}{q}}\right)
e^{i\frac{\pi(n-1)p}{q}}e^{-i\frac{n-1}{4}\pi}
+O(q)\\
=&\omega_q \sqrt q\, e^{2\pi i \frac{r(n-1)^2+rp^2}{q}}
\left(\sum_{\ell=0}^{\big\lfloor\frac N q\big\rfloor-1} e^{i\ell q\eta}\right)
e^{-i\frac{n-1}{4}\pi}+O(q).
\end{align*}
A similar argument shows that, for the same $\theta$,
$$S_-\Big(\frac{2\pi}{q},\theta\Big)\, e^{-i\phi_n(\theta)}
=\omega_q \sqrt q\, e^{2\pi i \frac{r(n-1)^2+rp^2}{q}}
\left(\sum_{\ell=0}^{\big\lfloor\frac N q\big\rfloor-1} e^{-i\ell q\eta}\right)
e^{i\frac{n-1}{4}\pi}+O(q).$$
Thus, by \eqref{eq:S_pm},
\begin{align*}
\left|S\Big(\frac{2\pi}{q},\theta\Big)\right|=\sqrt q
\left|\sum_{\ell=0}^{\big\lfloor\frac N q\big\rfloor-1} \cos\left(\frac{n-1}{4}\pi-\ell q\eta\right)\right|
+O(\sqrt N).
\end{align*}
Since $|\ell q\eta|\le \frac\pi 8$, the cosine's in the sum are of the same sign and satisfy
\begin{align*}
\Big|\cos\left(\frac{n-1}{4}\pi-\ell q\eta\right)\Big|
\ge C
\begin{cases}
     \ell q |\eta|&  \text{if } n\equiv 3 \text{ or } 7 \text{ (mod 8)} , \\
     1&  \text{otherwise.}
\end{cases}
\end{align*}
It follows that if $|\eta|\ge \frac{\pi}{16 N}$, then
$$\left|S\Big(\frac{2\pi}{q},\theta\Big)\right|\ge C \sqrt q\cdot\frac{N}{q}\approx N^{\frac34},$$
and, consequently, $Mf_N(\theta)\ge C N^{\frac34}$.

Now consider the set
$$E_N =\bigcup_{q \text{ odd: } \sqrt N\le q\le 2\sqrt N
\atop p \text{ even: } 2\varepsilon<\frac{2\pi p}{q}<\pi -2\varepsilon}
\left(\frac{2\pi p}{q}+\frac{\pi}{16N},\frac{2\pi p}{q}+\frac{\pi}{8N}\right).
$$
By the argument above,
$$Mf_N(\theta)\ge C N^{\frac34},\ \forall \theta\in E_N.$$
Since any two intervals in the definition of $E_N$ are either disjoint or identical,
by a counting argument (see \cite{MoyuaVega2008}), it follows that $|E_N|\ge C$.
This proves the proof of the   necessity  part  of Theorem \ref{Thm1.2}.
\end{proof}

 \noindent
{\bf Acknowledgments.}
X. Chen was supported by the NNSF of China, Grant Nos. 11901593 and  12071490.
 X.T. Duong was supported by  the Australian Research Council (ARC) through the research
grant DP190100970.
S. Lee  was supported by NRF (Republic of Korea) grant No. NRF2018R1A2B2006298.
 L. Yan was supported by the NNSF of China, Grant
No. 11521101 and 11871480, and by the Australian Research Council (ARC) through the research
grant DP190100970.

 \end{document}